\documentclass[12pt]{amsart}
\usepackage{bm,amsmath,amssymb,eucal}

\usepackage{hyperref}

\makeatletter
\@namedef{subjclassname@2020}{%
  \textup{2020} Mathematics Subject Classification}
\makeatother
\newtheorem{Th}{Theorem}[section]

\newtheorem{Lem}[Th]{Lemma}

\newtheorem{Fact}[Th]{Fact}

\theoremstyle{definition}

\newtheorem{Defs}[Th]{Definitions}

\newtheorem{Ex}[Th]{Example}

\theoremstyle{remark}

\newcommand\SP{\mathcal P}

\newcommand\N{\mathbb N}

\newcommand\bF{\boldsymbol F}

\newcommand\bH{\boldsymbol H}

\newcommand\bT{\boldsymbol T}

\newcommand\bde{\boldsymbol e}

\newcommand\bx{\boldsymbol x}
\newcommand\by{\boldsymbol y}

\newcommand\dbx{d_{\bx}}
\newcommand\dby{d_{\by}}
\newcommand\dbz{d_{\bz}}
\newcommand\dbxy{d_{\bx+\by}}

\newcommand\bz{\boldsymbol z}
\newcommand\ep{\varepsilon}

\newcommand\la{\lambda}

\newcommand\dl{\delta}

\newcommand\hT{\hat T}

\newcommand\bmu{\boldsymbol\mu}

\newcommand\dbmu{d_{\bmu}}

\newcommand\sbs{\subset}

\newcommand\leqs{\leqslant}
\newcommand\geqs{\geqslant}

\newcommand\oo{\infty}
\newcommand\loo{\ell_{\oo}}

\newcommand\lone{\ell_1}

\newcommand\supp{\operatorname{supp}}

\newcommand\nonat{\operatorname{nonat}}

\newcommand\dbf{d_{\bF}}

\newcommand\nor{\|{\cdot}\|}

\begin{document}
\title[Nonatomicity in $\lone$]{Nonatomic submeasures on $\N$ and the Banach space $\lone$}

\author[L.~Drewnowski]{Lech Drewnowski}
\address{Faculty of Mathematics and Computer Science\\
A.~Mickiewicz University\\
Uniw.Pozna\'nskiego 4, 61--614 Pozna\'n, Poland}
\email{drewlech@amu.edu.pl}
\date\today


\begin{abstract}
Nonatomic bounded sequences in $\lone$, that is, those giving rise to nonatomic submeasures on $\N$ are introduced and shown to form a closed subspace $\nonat(\lone)$ of $\loo(\lone)$, and some spaces of relevant operators in $\lone$ are considered.
\end{abstract}
\maketitle
\section{ Introduction}
This note may be considered a companion to the earlier joint works of T. {\L}uczak, N. Alon and the present author (\cite{DL1}, \cite{DL2},\cite{DL3},\cite{ADL}), dealing with nonatomic submeasures (or densities) on $\N$, more precisely - on its powerset $\SP(\N)$. Cf. also \cite{FP}, a complement to \cite{DL2}. For a comprehensive and yet fully sufficient general information on submeasures on $\N$, we recommend the reader \cite{DL2}. Our aim here is to provide a functional-analytic interpretation of the most promising constructions of nonatomic submeasures in terms of the Banach space $\lone$, and with the help of the results proven get some insight into the nature of the main open problem in this area.
The papers cited above were focused on submeasures on $\N$ of the type
$\dbmu=\limsup_{n\to\oo}\mu_n$ (understood setwise) for bounded sequences $\bmu=(\mu_n)$ of finite positive measures on $\N$. Now, any such a measure $\mu$ is intimately related to a convergent positive series $\sum_j\xi_j$ or to a positive element $x=(\xi_j)$ of the Banach lattice $\lone$.
A moment's reflection reveals that one may actually work with arbitrary elements $x\in \lone$, using, however, the modulus $|x|=(|\xi_j|)$ of $x$ to define a desired measure $\mu$ on $\N$. In the next step,  one makes a connection with the definition of the standard norm $\nor_1$ of $\lone$, and realizes that $\mu$ can also be obtained by the formula $\mu(A)=\|1_A x\|_1$ ($A\sbs\N$). Here $1_A$ stands for the characteristic function of the set $A$, and the associated \textit{characteristic} projection $P_A:x\mapsto 1_A x$ in $\lone$ has implicitly entered the play.
All this leads to the conclusion that the $\limsup$ submeasures described above can be realized as $\dbx$, where $\dbx(A)=\limsup_{n\to\oo}\|1_A x_n\|_1$, and $\bx=(x_n)$ is a bounded sequence in $\lone$, in other words, an element of the Banach space $\loo(\lone)$ with its obvious supnorm $\|\bx\|_\oo=\sup_n\|x_n\|_1$.
In these terms, the main question of interest in those papers may be phrased as follows: For what $\bx\in \loo(\lone)$ is the submeasure $\dbx$ nonatomic, that is, has the property that for every $\ep>0$ there exists a finite partition $(A_i:i\in I)$ of $\N$ (often called an $\ep$-\textit{fine} partition of $\N$ for $\bx$) such that $\dbx(A_i)\leqs\ep$ for each $i\in I$. When it is so, the sequence $\bx$ itself will simply be called \textit{nonatomic}. We denote the set of all such $\bx$ by $\nonat(\lone)$. Actually, the authors of \cite{DL1} considered the above mentioned problem as too difficult, and turned their attention to an apparently simpler case of bounded sequences in $(\lone)_{\rm fin}:=\{x\in \lone:|\supp(x)|<\oo\}$ (obviously a dense subspace of $\lone$), i.e., elements of $\loo((\lone)_{\rm fin})$, especially those similar to that used in the definition of the classic \textit{upper density} $\bar{d}$ on $\N$ (namely $\bx=(x_n)$ with $x_n=n^{-1}1_{[n]}$), a standard tool in number theory. As an argument speaking for such a reduction one may treat the following.
\begin{Fact}\label{fact:finsupp}
For any sequence $\bx=(x_n)\in l_\oo(\lone)$, and any sequence $(\ep_n)$ of positive numbers, there is a sequence $\bz=(z_n)\in l_\oo((\lone)_{\rm fin})$ such that, for each $n$, $\|x_n-z_n\|_1<\ep_n$. Consequently, when $\ep_n\to 0$, $\dbx=\dbz$, and if one is nonatomic, so is the other. In particular, $\nonat((\lone)_{\rm fin})$ is dense in $\nonat(\lone)$.
\end{Fact}
In order to convince the reader that nonatomic submeasures on $\N$ are not a rarity, we now present, in a  somewhat simplified form, the main result, Theorem 2.1 from \cite{DL1}, proved there by elementary, though worthy of attention probabilistic arguments. For this, some additional notation is needed.
\begin{Defs}\label{LT-defs}
For a finite nonempty set $F\sbs\N$, we let $x_F$ denote the norm one element of $\lone$ defined by $x_F=m^{-1}1_F = \sum_{j\in F}m^{-1}e_j$, where $m=|F|$.
Next, for each $\la>0$, we let $\bT_{\la}$ denote the set of all sequences $\bF=(F_n)$ of nonempty finite sets in $\N$ such that $q_m(\bF):=|\{n:|F_n|=m\}|\leqs 2^{\la m}$ for  each $m\in\N$. We will also write $\bx_{\bF}$ for the sequence $(x_{F_n})_n\in \loo(\lone)$.
\end{Defs}
\begin{Th} \label{LT:th main}
Every sequence $\bF=(F_n)\in \bT_\la$ ($\la>0$) yields a nonatomic submeasure $\dbf$ on $\N$ via the formula
$$
\dbf(A) =d_{\bx_{\bF}}(A)=\limsup_{n\to\oo}\|1_Ax_{F_n}\|_1 =\limsup_{n\to\oo}\dfrac{|A\cap F_n|}{|F_n|}.
$$
\end{Th}
 The question of a clear workable description of $\nonat(\lone)$ remains the most challenging open problem in this area. The results of this paper, though by no means decisive, may hopefully be of some help in a better understanding of the situation.
\section{The results}
Our two main results are: the first, appearing  right now, and the second, after a break for some preparations, stated as Theorem \ref{th:Tonlone} below.
\begin{Th}\label{th:nonatloolone}
 $N=\nonat(\lone)$ is a closed subspace of $\loo(\lone)$ containing a linear isometric copy of the Banach space $\loo$. Moreover, $N$ is solid in $\loo(\lone)$ in the usual sense, as well as in the sense of (equivalent!) relations of domination ($\ll$) or $0$-domination ($\ll_0$) transferred in a natural manner from the realm of submeasures to the space $\loo(\lone)$. In particular, $N$ is $\loo$-stable: $(a\in\loo,\bx\in N \implies a\bx\in N)$.
\end{Th}
Let us explain that `solid in the usual sense' means that $\by\in\nonat(\lone)$ when $\by\in\loo(\lone)$ and $|\by|\leqs|\bx|$ for some $\bx\in\nonat(\lone)$. That $N$ has this property is obvious. The other variants looks formally very similar, but with $|\by|\leqs|\bx|$ replaced by $\dby\ll\dbx$ or $\dby\ll_0\dbx$, resp. The first means that for every $\ep>0$ there is $\dl>0$ such that $\dby(A)\leqs\ep$ whenever $\dbx(A)\leqs\dl$, while the second that $\dby(A)=0$ whenever $\dbx(A)=0$. What we assert in the last part of the theorem  is clear in the case of $\ll$, while in the case of $\ll_0$ an appeal to \cite[Cor. 2.7]{DL2} is needed.
\begin{proof} In view of the above explanations, we only prove the principal part.
Let $\bx=(x_n),\by=(y_n)\in N$. It is obvious that $\dbxy\leqs \dbx+\dby$. Next, given $\ep>0$, choose $\ep/2$-fine finite partitions $(A_i:i\in I)$ and $(B_j:j\in J)$ of $\N$ for $\bx$ and  $\by$, resp. Clearly, also the sets $C_k:=A_i\cap B_j$ for $k=(i,j)\in K=I\times J$ form a finite partition of $\N$. Moreover, for each such $k$ one has $\dbxy(C_k)\leqs \dbx(C_k)+\dby(C_k)\leqs\dbx(A_i)+\dby(B_j)\leqs\ep$. In consequence, $\bx+\by\in N$. Now, for  any scalar $a$, choose any $m\in\N$ $\geqs|a|$. Then $m\bx\in N$ by the previous part, and since $|a\bx\leqs|m\bx|$, $a\bx\in N$.
In turn, let $\bz=(z_n)$ belong to the closure of $N$. Again, take any $\ep>0$. Next, choose $\bx\in N$ with $\|\bz-\bx\|_\oo\leqs\ep/2$, and an $\ep/2$-fine partition for $\bx$ as above. Then $\dbz\leqs\dbx+\ep/2$, and hence $\dbz(A_i)\leqs \ep/2+\ep/2=\ep$ for each $i\in I$. Thus $\bz\in\nonat(\lone)$. Finally, let $\bH=(H_n)$ be a sequence of finite nonempty sets in $\N$ such that $h_n=|H_n|\to\oo$. Also, let $\bx_{\bH}=(x_{H_n})$ be defined as in Definition \ref{LT-defs}. Then the map $a=(a_n)\mapsto a \bx_{\bH}=(a_n x_{H_n})$  is clearly a linear isometric embedding of $\loo$ into $\nonat(\lone)$.
The proof is over.
\end{proof}

\begin{Fact}
Every sequence $\bx=(x_n)\in\nonat(\lone)$ converges to zero uniformly on $\N$; that is, $\|x_n\|_\oo\to 0$.
\end{Fact}
\begin{proof}
Given $\ep>0$, there exists an $\ep$-fine finite partition $(A_i:i\in I)$ of $\N$ for $\bx$. From this it follows that for some $m$ and all $n\geqs m$ one has $\|P_{A_i} x_n\|_\oo\leqs \|P_{A_i} x_n\|_1\leqs\ep$. In consequence, $\|x_n\|_\oo =\max_{i\in I} \|P_{A_i} x_n\|_\oo\leqs \ep$ for all $n\geqs m$, and we are done.
\end{proof}
Before proceeding, we introduce some natural subspaces of $\loo(\lone)$, situated in a close vicinity of $\nonat(\lone)$. We shall denote by $\loo(\lone,\overset{p}{\to}0)$ and $\loo(\lone,\overset{u}{\to}0)$ the closed subspaces of $\loo(\lone)$ consisting of sequences $\bx=(x_n)$ that converge to zero pointwise on $\N$) or uniformly on $\N$, respectively. The chain of inclusions displayed below may help the reader to visualize the situation.
$$
c_0(\lone)\sbs\nonat(\lone)\sbs \loo(\lone,\overset{u}{\to}0)\sbs \loo(\lone,\overset{p}{\to}0)\sbs \loo(\lone).
$$
The second inclusion from the left was shown in the Fact above, the other are trivial, and all are proper. That it is so for the second one can be seen from \cite[Th. 2.2]{DL1}, where a sequence $\bF=(F_n)$ of finite nonempty subsets of $\N$
was constructed such that for any $m$, $|F_n|=m$ for finitely many $n$'s only, and thus $\|x_{F_n}\|_\oo\to 0$, but the submeasure $\dbf$ is not nonatomic; in other terms, the sequence $\bx =(x_{F_n})$ is not nonatomic.
In what follows we will denote by $\bde=(e_n)$ the standard basis of $\lone$, and by $L(\lone)$ the usual space of (continuous linear) operators acting in $\lone$.
Furthermore, for any $T\in L(\lone)$ we let $\hT$ denote the associated operator acting in $\loo(\lone)$ coordinatewise: $\hT\bx=(Tx_n)$ for every $\bx=(x_n)\in\loo(\lone)$. The lemma below incorporates most of the technicalities needed to prove our result.
\begin{Lem}
Let $T\in L(\lone)$ and denote $\bz=(z_n)=\hT\bde$. Further, let $\bx=(x_n)\in \loo(\lone,\overset{p}{\to}0)$ and denote $\by=(y_n)=\hT\bx$. Moreover, denote $t=\|\bz\|_\oo =\sup_n \|z_n\|_1$ and $r=\|\bx\|_\oo=\sup_n \|x_n\|_1$, and let $A\sbs\N$.
\begin{enumerate}
\item
$\limsup_n\|P_A y_n\|_\oo\leqs r\limsup_n\|P_A z_n\|_\oo$.
Hence if the sequence $\bz$ converges to zero uniformly on $A$, that is, $\|P_Az_n\|_\oo\to 0$, so does the sequence $\by$.
\item
$\dby(A)\leqs r\dbz(A)$. Consequently, $\dby(A)=0$ whenever $\dbz(A)=0$; moreover, for any $\dl>0$, every $\dl$-fine partition of $\N$ for $\bz$ (if such exists) is  $r\dl$-fine for $\by$.
\end{enumerate}
\end{Lem}
\begin{proof}
For the sake of clarity and convenience, we shall use the notation $x=(x(j))$ for the elements $x$ in $\lone$ as the domain space of $T$, and $y=(y(k))$ for the elements $y$ of $\lone$ as its range space (in either case $j,k\in\N$). Moreover, the index $n$ in $(e_n)$ and $(z_n)$ will sometimes be changed to $j$. Clearly, $T$ is represented by the sequence $\bz$ so that for every $x=(x(j))=\sum_jx(j)e_j$ in $\lone$ one has $Tx=\sum_j x(j)z_j$ (the series being absolutely convergent in $\lone$ in its default norm $\nor_1$, hence also in the weaker norm $\nor_\oo$ induced from $c_0$). Hence $(Tx)(k)=\sum_j x(j)z_j(k)$ for each $k$. Also observe that $t=\|T\|$ and that $P_ATx=\sum_j x(j)P_Az_j$ ($A\sbs\N$). In some estimates below we use the fact that $\|P_A\|\leqs 1$.
Fix $\ep>0$, and next choose $\dl>0$ with $r\dl\leqs\ep/2$.
\begin{enumerate}
\item:  For $s:=\limsup_n\|P_A z_n\|_\oo$ choose $m$ so that $\|P_A z_n\|_\oo\leqs s+\dl$ for $n>m$; note that $ \|P_Az_n\|_\oo\leqs\|P_A z_n\|_1\leqs\|z_n\|_1\leqs t$ for $n\in [m]$. Also, as $x_n\overset{p}{\to}0$, there is $p$ such that $t\|1_{[[m]}x_n\|_1\leqs\ep/2$ for all $n\geqs p$. Now, for all $n\geqs p$ one has
    $$
    \begin{aligned}
    \|P_Ay_n\|_\oo &\leqs\sum_j |x_n(j)| \|P_A z_j\|_\oo\\
    &= \sum_{j> m} |x_n(j)| \|P_A z_j\|_\oo + \sum_{j\in[m]} |x_n(j)| \|P_A z_j\|_\oo\\
    &\leqs \|x_n\|_1 (s+\dl) + t\|1_{[[m]}x_n\|_1\leqs r(s+\dl)+\ep/2\leqs rs+\ep,
    \end{aligned}
    $$
    and the inequality asserted in $(1)$ follows.
\item: Here the estimates are similar. Start by choosing $m$ so that $\|P_Az_n\|_1\leqs \dbz(A)+\dl$ for all $n> m$; note that  $\|P_Az_n\|_1\leqs\|z_n\|_1\leqs t$ for $n\in[m]$ . Next, choose $p$ as in the part above. Then for all $n\geqs p$,
    $$
    \begin{aligned}
    \|P_Ay_n\|_1 &\leqs\sum_j |x_n(j)| \|P_A z_j\|_1\\
    &= \sum_{j> m} |x_n(j)| \|P_A z_j\|_1 + \sum_{j\in[m]} |x_n(j)| \|P_A z_j\|_1\\
    &\leqs \|x_n\|_1 (\dbz(A)+\dl) + t\|1_{[[m]}x_n\|_1\\
    &\leqs r(\dbz(A)+\dl)+\ep/2\leqs r\dbz(A)+\ep,
    \end{aligned}
    $$
    and  the desired inequality follows.
\end{enumerate}
\end{proof}
 \begin{Th}\label{th:Tonlone}
 Let $T\in L(\lone)$, and let $\bz=\hT\bde$.
 \begin{enumerate}
 \item If $\bz\in \loo(\lone,\overset{p}{\to}0)$, then $\hT$ maps $\loo(\lone,\overset{p}{\to}0)$ into itself.
 \item If $\bz\in \loo(\lone,\overset{u}{\to}0)$, then $\hT$ maps $\loo(\lone,\overset{p}{\to}0)$ into $\loo(\lone,\overset{u}{\to}0)$.
 \item If $\bz\in \nonat(\lone)$, then $\hT$ maps $\loo(\lone,\overset{p}{\to}0)$ into $\nonat(\lone)$; in particular, $\hT$ maps $\nonat(\lone)$ into itself.
 \end{enumerate}
 In either case, $\hT\bx\ll\bz$ for all $\bx\in \loo(\lone,\overset{p}{\to}0)$.
Finally, the operators $T$ appearing in the assumptions or assertions of $(1)$--$(3)$ form closed subspaces of $L(\lone)$.
 \end{Th}
 \begin{proof}

 $(1)$ \& $(2)$:\quad Apply part $(1)$ of the lemma with $A=\{k\}$ ($k\in\N$) or $A=\N$.

$(3)$:\quad Apply part $(2)$ of the lemma. The final statement is easy to verify.
 \end{proof}
 Theorem \ref{th:Tonlone}\,$(3)$ can be used as a tool to produce new nonatomic sequences out of those already known. A simple illustration of that is provided below.
 \begin{Ex}
 Let $\bx=(x_n)\in \nonat(\lone)$, let $T\in L(\lone)$ be determined by the condition  $\hT\bde=\bx$, and let $\bH=(H_n)$ be a sequence of finite nonempty sets in $\N$ such that $h_n=|H_n|\to\oo$. Also, let $\bx_{\bH}=(x_{H_n})\in \loo(\lone,\overset{u}{\to}0)$ be defined as in Definition \ref{LT-defs}. Then
 $$\by=\hT\bx_{\bH}=\bigl(h^{-1}_n\sum_{j\in H_n}x_j\bigr)_n\in \nonat(\lone).$$
 Notice that the choice $H_n=[n]$ ($n\in\N$) yields the simplest case: $(n^{-1}\sum_{j=1}^n x_j)_n \in \nonat(\lone)$.
 \end{Ex}

\end{document}